\documentclass[12pt]{article}   	% use "amsart" instead of "article" for AMSLaTeX format
\usepackage{geometry}                		% See geometry.pdf to learn the layout options. There are lots.
\usepackage[parfill]{parskip}    		% Activate to begin paragraphs with an empty line rather than an indent
\usepackage{graphicx}				% Use pdf, png, jpg, or eps§ with pdflatex; use eps in DVI mode
\usepackage[pdftex,pagebackref,letterpaper=true,colorlinks=true,pdfpagemode=none,urlcolor=blue,linkcolor=blue,citecolor=blue,pdfstartview=FitH]{hyperref}
\usepackage{amsmath,amsfonts}
\usepackage{color}
\usepackage{amsthm}
\usepackage{boolexpr}
\usepackage{amssymb,latexsym,pstcol,pst-plot}
\usepackage[UKenglish]{isodate}% http://ctan.org/pkg/isodate
\allowdisplaybreaks

\usepackage{amssymb}
\newtheorem{thm}{Theorem}
\newtheorem{cor}[thm]{Corollary}
\newtheorem*{remark}{Remark}

\title{Asymptotic density of Motzkin numbers modulo small primes}
\author{Rob Burns}
\begin{document}
\maketitle
\begin{abstract}
We establish the asymptotic density of the Motzkin numbers modulo $2$, $4$, $8$, $3$ and $5$.
\end{abstract}

\section{Introduction}
%\section{}
%\subsection{}
The Motzkin numbers $M_n$ are defined by
$$
M_n := \sum_{k \geq 0} \binom {n}{2k} C_{k}
$$
where $C_{k}\, $ are the Catalan numbers. 

There has been some work in recent years on analysing the Motzkin numbers $M_n$ modulo primes and prime powers.  This work has often been done in concert with and using the same methods as work analysing the Catalan numbers. Deutsch and Sagan \cite{Sagan2006} provided a characterisation of Motzkin numbers divisible by $2, 4$ and $5$. They also provided a complete characterisation of the Motzkin numbers modulo $3$ and showed that no Motzkin number is divisible by $8$. Eu, Liu and Yeh \cite{Eu2008} reproved some of these results and extended them to include criteria for when $M_n$ is congruent to $\{ 2, 4, 6 \} \mod 8$. Krattenthaler and M{\"u}ller \cite{KM2013} established identities for the Motzkin numbers modulo higher powers of 3 which include the modulo 3 result of  \cite{Sagan2006} as a special case. Krattenthaler and M{\"u}ller \cite{Krat2016} have more recently extended this work to a full characterisation of \mbox{$M_n \mod 8$} in terms of the binary expansion of $n$. Their characterisation is rather elaborate and less susceptible to analysis than that provided in \cite{Eu2008}. The results in \cite{KM2013} and \cite{Krat2016} are obtained by expressing the generating function of $M_n$ as a polynomial involving a special function. Rowland and Yassawi \cite{RY2013} investigated $M_n$ in the general setting of automatic sequences.  The values of $M_n$ (as well as other sequences) modulo prime powers can be computed via automata. Rowland and Yassawi provided algorithms for creating the relevant automata. They established results for $M_n$ modulo small prime powers, including a full characterisation of $M_n$ modulo 8 (modulo $5^2$ and $13^2$ are available from Rowlands website). They also established that $0$ is a forbidden residue for $M_n$ modulo $8$, $5^2$ and $13^2$. In theory the automata can be constructed for any prime power but computing power and memory quickly becomes a barrier. For example, the automata for $M_n$ modulo $13^2$ has over 2000 states. Rowland and Yassawi also went on to describe a method for obtaining asymptotic densities of $M_n$.

We will use the above results to establish asymptotic densities of $M_n$ modulo $2, 4, 8, 3$ and $5$. Here, the asymptotic density of a subset $S$ of $\mathbb{N}$ is defined to be
$$
\lim_{N \to \infty} \frac {1}{N} \#\{ n \in S : n \leq N \}
$$
if the limit exists, where $\#S$ is the number of elements in a set $S$. In contrast to the results for the Catalan numbers $C_n$, the set of Motzkin numbers congruent to $0 \mod n$ is not expected to have asymptotic density $1$ for a general $n \in \mathbb{N}$. The results here show that this expectation holds for small values of $n$.

\section{Asymptotic density of certain forms of numbers}
The main method in the literature of characterising $M_n \mod q$ is to divide the natural numbers into classes of the form
$$
S(\, q, r, s, t ) = \{ (\, qi + r )\, q^{sj + t} + c: i, j \in \mathbb{N} \} 
$$
for various choices of $r, s, t$ and $c$. It will therefore be useful to know how these types of sets behave asymptotically. We can disregard the $c$ term as this does not change the asymptotic behaviour. So the set of interest is
\begin{equation}
\label{S}
S(\, q, r, s, t )\, = \{ (\, qi + r )\, q^{sj + t} : i, j \in \mathbb{N} \}
\end{equation}
for integers $q, r, s, t$.
\bigskip

\begin{thm}
\label{a}
Let $q, r, s, t \in \mathbb{Z} $ with $q, s > 0$, $t \geq 0$ and $0 \leq r < q$. Then the asymptotic density of the set $S$ is $(\, q^{t + 1 - s} (\, q^{s} - 1 )\, )\, ^{-1} $. 
\end{thm}
\begin{proof}
We have, for fixed $j \geq 0$,
$$
\#\{ i \geq 0: (\, qi + r)\, q^{sj + t} \leq N \} = \frac {N}{q^{sj+t+1}} - \frac {r}{q} - E(\, j, N, q, r, s, t )
$$
where $0 \leq E( j, N, q, r, s, t )  < 1$ is an error term introduced by not rounding down to the nearest integer. So, letting
$$
U(N, s, t) := \lfloor  \frac {\log_q (\, N )\, - t - 1}{s} \rfloor 
$$,
we have
$$
\mbox{$\#\{n < N: n = (\, qi + r)\, q^{sj + t} \,$ for some   $i, j \in \mathbb{N} \}$}
$$
$$
= \sum_{j \geq 0} (\, \frac {N}{q^{sj + t + 1}}  - \frac {r}{q} - E(\, j, N, q, r, s, t )\, )\, 
$$
$$
= \sum_{j = 0}^{U} (\, \frac {N}{q^{sj+t + 1}}  - \frac {r}{q} - E(\, j, N, q, r, s, t )\, )\,
$$
$$
= \frac {N}{q^{t + 1} } \sum_{j = 0}^{U} (\, \frac {1}{q^{s}} )\,^{j} - E^{'}(\, N, q, r, s, t )\,
$$
where the new error term $E^{'} (\, N, q, r, s, t )\,$ satisfies 
$$
0 < E^{'} ( N, q, r, s, t ) < 2( U + 1 ) .
$$
Then
$$
\mbox{ $\#\{n < N: n = (\, qi + r)\, q^{sj + t}$  for some   $i, j \in \mathbb{N} \}$ }
$$
$$
= (\, \frac {N}{q^{t+1}} )\, (\, 1 - (\, \frac {1}{q^{s}} )\, ^{U + 1} )\,  (\, 1 - \frac {1}{q^{s} } )\, ^{-1} - E^{'}.
$$
Since $\lim_{N \to \infty} \frac {E^{'}(\, N, q, r, s, t )\, }{N} = 0$ and $\lim_{N \to \infty} \frac {1}{N} (\, \frac {1}{q^{s} } )\, ^{U + 1} = 0$ we have
$$
\mbox{ $\lim_{N \to \infty} \frac {1}{N} \#\{n < N: n = (\, qi + r)\, q^{sj + t}$  for some   $i, j \in \mathbb{N} \}$}
$$
$$
= (\, q^{t + 1 - s} (\, q^{s} - 1 )\,)\, ^{-1}.
$$
\end{proof}

\bigskip

\begin{remark}
\label{remS'}
Sometimes we will need to consider the set
\begin{equation}
\label{S'}
S^{'}(\, q, r, s, t )\, = \{ (\, qi + r )\, q^{sj + t} : i, j \in \mathbb{N}, j \geq 1 \}
\end{equation}
for integers $q, r, s, t$.
The asymptotic density of the set $S^{'}$ can be derived from theorem~\ref{a} as \mbox{$(\, q^{t + 1} (\, q^{s} - 1 )\, )\, ^{-1} $}.
\end{remark}
\bigskip

\section{Motzkin numbers modulo $2$, $4$ and $8$}
\label{mod2^k}
The following result is established in \cite{Eu2008}
\begin{thm}
\label{5.5}
(Theorem 5.5 of \cite{Eu2008}). The $n$th Motzkin number $M_n$ is even if and only if 
$$
\mbox{$n = (4i + \epsilon)4^{j+1} - \delta$ for $i, j \in \mathbb{N}, \epsilon \in \{1, 3\}$ and $\delta \in \{1, 2\}$. }
$$ 
Moreover, we have
$$
\mbox{$M_n \equiv 4 \mod 8$ if $(\, \epsilon, \delta )\, = (\, 1, 1 )\,$ or $(\, 3, 2 )\,$ }
$$
$$
\mbox{$M_n \equiv 4y + 2 \mod 8 $ if $(\, \epsilon, \delta )\, = (\, 1, 2 )\, or (\, 3, 1 )\,$ }
$$
where $y$ is the number of digit $1$Õs in the base 2 representation of $4i + \epsilon - 1$.
\end{thm}
\bigskip

\begin{remark}{}
The $4$ choices of $(\, \epsilon, \delta )\,$ in the above theorem give $4$ disjoint sets of numbers \mbox{$n = (4i + \epsilon)4^{j+1} - \delta$}.
\end{remark}

\bigskip
\begin{thm}
\label{b}
Each of the 4 disjoint sets defined by the choice of $(\, \epsilon, \delta )\,$ in Theorem~\ref{5.5} has asymptotic density $\frac {1}{12}$ in the natural numbers.
\end{thm}
\begin{proof}
Use the result of Theorem~\ref{a} for the set $S$  with $q = 4, r = \epsilon, s = 1, t = 1$.
\end{proof}
\bigskip
\begin{cor}
The asymptotic density of
$$
\mbox{$\{ n < N: M_n \equiv 0 \mod 2 \}$ is $\frac {1}{3}$.}
$$
The asymptotic density of
$$
\mbox{$\{ n < N: M_n \equiv 4 \mod 8 \}$ is $\frac {1}{6}$.}
$$
The asymptotic density of each the sets
$$
\mbox{$\{ n < N: M_n \equiv 2 \mod 8 \}$ and $\{ n < N: M_n \equiv 6 \mod 8 \}$ is $\frac {1}{12}$.}
$$
\end{cor}
\begin{proof}
The first 2 statements of the corollary follow immediately from theorem~\ref{5.5} and theorem~\ref{b}. The third statement follows from the observation that the numbers of $1$'s in the base 2 expansion of $i$ is equally likely to be odd or even and therefore the same applies to the the number of $1$'s in the base 2 expansion of $4i + \epsilon - 1$. Since the asymptotic density of the 2 sets combined is $\frac {1}{6}$ (from theorem~\ref{b}), each of the two sets has asymptotic density $\frac {1}{12}$.
\end{proof}
\begin{remark}{}
Rowland and Yassawi \cite{RY2013} proved the first two results of the corollary and also established that the asymptotic density of the sets of $M_n$ congruent to $2$ modulo $4$ is $\frac{1}{6}$.
\end{remark}

\section{Motzkin numbers modulo 5}
The following result is established in \cite{Sagan2006}
\bigskip
\begin{thm}
\label{m5.1}
(Theorem 5.4 of \cite{Sagan2006}). The Motzkin number $M_n$ is divisible by $5$ if and only if $n$ is one of the following forms
$$
(\, 5i + 1 )\, 5^{2j}  - 2,\, (\, 5i + 2 )\, 5^{2j-1}  - 1,\, (\, 5i + 3 )\, 5^{2j-1}  - 2,\,  (\, 5i + 4 )\, 5^{2j}  - 1
$$
where $i, j \in \mathbb{N}$ and $j \geq 1$.
\end{thm}

\bigskip

\begin{thm}
\label{m5.2}
The asymptotic density of numbers of the first form in theorem~\ref{m5.1} is $\frac {1}{120}$. Numbers of the fourth form also have asymptotic density $\frac {1}{120}$. The asymptotic density of numbers of the second and third forms in theorem~\ref{m5.1} is $\frac {1}{24}$ each.
\end{thm}
\begin{proof}
Firstly consider numbers of the form $(\, 5i + r)\, 5^{2j} - 2\,$. As we are interested in asymptotic density it is enough to look at numbers of the form $(\, 5i + r)\, 5^{2j}\,$. We can now use the remark~\ref{remS'} at the end of theorem~\ref{a} for the set $S^{'}$ with $q=5, s=2$ and $t=0$. From the remark the asymptotic density of the set
$$
\mbox{$\{ n \in \mathbb{N}: n = (\, 5i + r )\, 5^{2j}$ with $i, j \in \mathbb{N}$ and $j \geq 1$ \} }
$$
is $(5 \times (5^2 - 1) )^{-1} =  \frac{1}{120}$. For numbers of the second and third forms we shift the j index so that it starts from $0$ and use theorem~\ref{a} for the set $S$ with $q=5, s=2$ and $t=1$. From theorem~\ref{a} the asymptotic density of the set
$$
\mbox{$\{ n \in \mathbb{N}: n = (\, 5i + r )\, 5^{2j + 1}$ with $i, j \in \mathbb{N}$ and $j \geq 0$ \} }
$$ 
is $(5^0(5^2 - 1))^{-1} = \frac {1}{24}$.
\end{proof}

\begin{cor}
The asymptotic density of $\#\{n < N: M_{n} \equiv 0 \mod {5} \}$ is $\frac {1}{10} \, $.
\end{cor}
\begin{proof}
The corollary follows immediately from theorem~\ref{m5.1} and theorem~\ref{m5.2} and the disjointness of the 4 forms of integers listed in theorem~\ref{m5.1}.
\end{proof}

\begin{remark}{}
Numerical tests also show that roughly 22.5\% of Motzkin numbers are congruent to each of $1, 2, 3, 4 \mod {5}$.
\end{remark}

\section{Motzkin numbers modulo 3}
The structure of the Motzkin numbers modulo $3$ is based on a set $T(01)$ which was defined by Deutsh and Sagan in \cite{Sagan2006}. The set $T(\, 01 )\, $ is the set of natural numbers which have a base 3 expansion containing only the digits $0$ and $1$. The following theorem from \cite{Sagan2006} will be used in this section.
\bigskip

\begin{thm}
\label{mod3-1}
(Corollary 4.10 of \cite{Sagan2006}). The Motzkin numbers satisfy
\begin{align*}
M_n &\equiv  -1 \mod 3 \quad if  \quad n \in 3T (\, 01 )\, - 1, \\
M_n &\equiv 1 \mod 3  \quad if  \quad n \in 3T(\, 01 ) \quad or \quad n \in 3T(\, 01 )\, - 2, \\
M_n &\equiv 0 \mod 3 \quad otherwise.
\end{align*}
\end{thm}
\bigskip

We will first examine the nature of the set $T(\, 01 \,)$. We have,
\bigskip

\begin{thm}
\label{mod3-2}
The asymptotic density of the set $T(\, 01 \,)$ is zero.
\end{thm}
\begin{proof}
Let $N \in \mathbb{N}$ and choose $k \in \mathbb{N}: 3^{k} \leq N < 3^{k+1}$. Then $k = \lfloor log_3 (\, N )\, \rfloor$ and
\begin{gather*}
\frac {1}{N} \# \{ n \leq N :  n \in  T(01) \}  \\
\leq \frac {2^{k+1}}{N}  \\
\mbox{$\leq \frac {2^{k+1}}{3^{k}} \to 0$ as $N \to \infty$.} 
\end{gather*}
\end{proof}
\bigskip

\begin{thm}
The asymptotic density of the set  \mbox{$\{n \leq N: M_n \equiv 0 \mod 3 \}$ } is $1$.
\end{thm}
\begin{proof}
Since the asymptotic density of $T(\, 01 \,)$ is zero, so is the asymptotic density of the sets $3T(\, 01 \,) - k$ for $k \in \{0, 1, 2\}$. Therefore theorem~\ref{mod3-1} implies that
$$
\lim_{N \to \infty} \frac {1}{N} \# \{n \leq N : M_n \equiv \,   _-^+ \, 1 \mod 3 \} = 0
$$
and the result follows.
\end{proof}

\bigskip
\bibliographystyle{plain}
\begin{small}
\bibliography{ref}

\begin{thebibliography}{1}

\bibitem{Sagan2006}
E.~Deutsch and B.E. Sagan.
\newblock Congruences for {C}atalan and {M}otzkin numbers and related
  sequences.
\newblock {\em Journal of Number Theory}, 117(1):191--215, 2006.

\bibitem{Eu2008}
Sen-Peng Eu, Shu-Chung Liu, and Yeong-Nan Yeh.
\newblock Catalan and {M}otzkin numbers modulo 4 and 8.
\newblock {\em European Journal of Combinatorics}, 29:1449--1466, 2008.

\bibitem{Krat2016}
C.~Krattenthaler and T.~W. M\"uller.
\newblock Motzkin numbers and related sequences modulo powers of $2$.
\newblock {\em ArXiv}, arXiv:1608.05657:28, 2016.

\bibitem{KM2013}
Christian Krattenthaler and Thomas~W. M\"uller.
\newblock A method for determining the mod-$3^k$ behaviour of recursive
  sequences.
\newblock {\em ArXiv}, http://arxiv.org/abs/1308.2856, 2013.

\bibitem{RY2013}
Eric Rowland and Reem Yassawi.
\newblock Automatic congruences for diagonals of rational functions.
\newblock {\em ArXiv}, https://arxiv.org/abs/1310.8635, 2013.

\end{thebibliography}
\end{small}

\end{document}